\documentclass[12pt]{amsart}
\usepackage{amssymb,latexsym, amsmath, amscd, array, graphicx}

\swapnumbers \numberwithin{equation}{section}

\theoremstyle{plain}

\newtheorem{thm}{Theorem}[section]

\newtheorem{conjec}[thm]{Conjecture}
\newtheorem{prop}[thm]{Proposition}
\newtheorem{cor}[thm]{Corollary}

\theoremstyle{definition}
\newtheorem{defin}[thm]{Definition}

 \newcommand{\Wi}{\widetilde}

\def\scr{\mathcal}

\def\Z{{\mathbb Z}}
\def\Q{{\mathbb Q}}
\def\R{{\mathbb R}}
\def\N{{\mathbb N}}

\def\1{\hbox{\rm\rlap {1}\hskip.03in{\rom I}}}
\def\Bbbone{{\rm1\mathchoice{\kern-0.25em}{\kern-0.25em}
{\kern-0.2em}{\kern-0.2em}I}}

\long\def\forget#1\forgotten{} %

\newcommand\ver[1]{\marginpar{\tiny Changed in Ver \VER}}

\date{\today}
\begin{document}

\title[On macroscopic dimension of rationally essential
manifolds]{On macroscopic dimension of rationally essential
manifolds }

\author[A.~Dranishnikov]{Alexander  Dranishnikov$^{1}$} %
\thanks{$^{1}$Supported by NSF, grant DMS-0904278}

\address{Alexander N. Dranishnikov, Department of Mathematics, University
of Florida, 358 Little Hall, Gainesville, FL 32611-8105, USA}
\email{dranish@math.ufl.edu}



\begin{abstract}
We construct a counterexamples in dimensions $n>3$ to Gromov's
conjecture \cite{Gr1} that the macroscopic dimension of rationally
essential $n$-dimensional manifolds equals $n$.
\end{abstract}

\maketitle \tableofcontents

\section {Introduction}

Gromov brought to life several definitions of a large scale
dimension. Perhaps the best known of these is the notion of the
asymptotic dimension introduced as an invariant of discrete groups
\cite{Gr3}. It turned out that the finiteness of  asymptotic
dimension for a group implies many famous conjectures of Novikov's
type for that group
\cite{Yu},\cite{Ba},\cite{CG},\cite{BR},\cite{DFW}. The asymptotic
dimension $asdim X$ is defined for general metric spaces $X$ and by
its nature does not take into account the local structure of a
space. The notion of  macroscopic dimension $\dim_{mc}X$ was
introduced by Gromov for studying Riemannian manifolds \cite{Gr1}.
This notion of large-scale dimension is sensitive to the local
structure in particular, to the covering dimension of a space which
by definition is local. We note that always $\dim_{mc}X\le asdim X$.

Gromov stated several conjectures and questions concerning
macroscopic dimension. One of his conjectures on $\dim_{mc}$ was
that the universal covering $\Wi M$ of any $n$-manifold $M$ with
positive scalar curvature satisfies the inequality $\dim_{mc}\Wi
M\le n-2$. This conjecture seems to be out of reach since it implies
the Gromov-Lawson conjecture about non-existence of a positive
scalar curvature metric on any closed aspherical manifold. The
latter is known to be a Novikov type conjecture. We refer to
\cite{BD} for recent progress on the Gromov scalar curvature
conjecture modulo the Novikov conjecture. In this paper we solve
another Gromov's problem which connects the macroscopic dimension of
the universal covering with the essentiality of the manifold.

\begin{defin}\cite{Gr1}
A metric space $X$ has the {\em macroscopic dimension} less or equal
to $k$, $\dim_{mc}X\le k$, if there is a continuous uniformly
cobounded map $f:X\to N^k$ to a $k$-dimensional simplicial complex.
\end{defin}
A map $f:X\to Y$ of a metric space is {\em uniformly cobounded} if
there is a constant $C>0$ such that $diam(f^{-1}(y))<C$ for all
$y\in Y$.

Clearly $\dim_{mc}$ is an invariant of quasi-isometry
homeomorphisms. Therefore, the macroscopic dimension $\dim_{mc}\Wi
M$ of the universal covering $\Wi M$ of a closed manifold $M$ taken
with the lifted from $M$ metric does not depend on the choice of the
metric on $M$. Gromov studied the question when the universal
covering $\Wi M$ of a closed $n$-dimensional manifold could have
macroscopic dimension strictly less than $n$.

The main source of $n$-manifolds satisfying the inequality
$\dim_{mc}\Wi M<n$  is the class of inessential manifolds also
introduced by Gromov \cite{Gr2}. We recall that an $n$-manifold $M$
is called {\em inessential} if a map $f:M\to B\pi$ that classifies
its universal covering $\Wi M$ can be deformed to the
$(n-1)$-skeleton $B\pi^{(n-1)}$. Otherwise it is called {\em
essential}. It is well-known that an orientable manifold is
inessential if and only if the image of its fundamental class under
the induced homomorphism $f_*:H_*(M)\to H_*(B\pi)$ is zero,
$f_*([M])=0$ (see~\cite{BD}). An example of an essential
$n$-manifold $M$ with $\dim_{mc}\Wi M<n$ is the real projective
space $\R P^n$. Though $M=\R P^n$ is  essential it is {\em
rationally inessential} for odd $n$, i.e., with $f_*([M])=0$ in
$H_*(B\pi;\Q)$. Gromov conjectured this  is always the case.
\begin{conjec}\label{GQ}(\cite{Gr1}, $2\frac{2}{3}$ Remarks and Questions)
If $\dim_{mc}\Wi M <n$ where $\Wi M$ is the universal covering of a
closed orientable $n$-manifold $M$, then $M$ must be rationally
inessential.
\end{conjec}
The main goal of this paper to give  counterexamples to Gromov's
conjecture. For that we define a new notion of macroscopic dimension
$\dim_{MC}$ satisfying the inequality
$$
\dim_{mc}X\le\dim_{MC}X\le asdim X$$ and construct  rationally
essential closed $n$-manifolds $M$ with $\dim_{MC}\Wi M<n$.
\begin{defin}
A metric space $X$ has the {\em macroscopic dimension} less than or
equal to $k$, $\dim_{MC}X\le k$, if there is a Lipschitz uniformly
cobounded map $f:X\to N^k$ to a $k$-dimensional simplicial complex.
\end{defin}
Here we assume that a simplicial complex $N$ has a metric inherited
from the Hilbert space $\ell_2(N^{(0)})$ spanned by the vertices of
$N$ under the canonical imbedding into the standard simplex
$N\subset\Delta\subset\ell_2(N^{(0)})$. We will call $\dim_{MC}$ the
macroscopic dimension and will refer to $\dim_{mc}$ as  Gromov's
macroscopic dimension. Clearly,
$$
\dim_{mc}X\le\dim_{MC}X.$$ The original definition of asymptotic
dimension uses coverings by large open sets. Alternatively, the
asymptotic dimension can be defined as follows \cite{Gr3},\cite{BD}:
\begin{defin}
A metric space $X$ has  {\em asymptotic dimension} less than or
equal to $k$, $asdim X\le k$, if for every $\epsilon>0$ there is an
$\epsilon$-Lipschitz uniformly cobounded map $f:X\to N^k$ to a
$k$-dimensional simplicial complex.
\end{defin}
Then, clearly,
$$\dim_{MC}X\le asdim X.$$

In this paper we develop a cohomological approach to  macroscopic
dimension outlined in \cite{Dr}. This theory combined with a
homological characterization of amenability given by Block and
Weinberger~\cite{BW} produces the following examples.
\begin{thm}\label{main*}
For all $n\ge 4$ there are closed rationally essential $n$-manifolds
$M$ with  $\dim_{MC}\Wi M<n$.
\end{thm}
The non-amenability of the fundamental group in these examples is
essential in view of the following
\begin{thm}\cite{Dr}\label{dr}
For rationally essential $n$-manifolds with amenable fundamental
group there is the equality $\dim_{MC}\Wi M=n$.
\end{thm}
It turns out that the inequality $\dim_{MC}\Wi M< n$ depends only on
the homology class $\alpha=f_*([M])\in H_n(B\pi)$ where $f:M\to
B\pi$ is a map that classifies the universal covering of $M$. It
means that for any manifold $N$ with classifying map $g:N\to B\pi$
and $g_*([N])=\alpha$ it follows that $\dim_{MC}\Wi N<n$. Following
Brunnbauer and Hanke \cite{BH} we call such classes {\em small}. It
turns out that small classes form a subgroup in $H_n(B\pi)$. This
phenomenon was discovered first in \cite{BH} with respect to many
classes of so-called large manifolds such as manifolds with
hyper-spherical universal cover, with hyper-euclidean universal
cover, enlargeable, etc.  The property $\dim_{MC}\Wi M=\dim M$ for a
manifold $M$ certainly represents some largeness. We call such
manifolds $M$ {\em macroscopic dimension large}. Nevertheless
Brunnbauer-Hanke approach does not work here. In particular, for all
largeness conditions treated in \cite{BH} the torsion elements are
small. It is still an open question whether torsion elements of
$H_*(B\pi)$ are small for the macroscopic dimension large property.
Our proof that small classes in $H_*(B\pi)$ for the macroscopic
dimension largeness  form a subgroup is based on a concept of the
almost equivariant (co)homology introduced in this paper. Curiously,
this proof brought as the byproduct the following formula for
homology of groups:
$$
H_n(B\pi)=ker\{I^n\otimes_{\pi}\Z\stackrel{(1\otimes j)\otimes
1}\longrightarrow (I^{n-1}\otimes\Z\pi)\otimes_{\pi}\Z\}$$ where
$j:I\to\Z\pi$ is the inclusion of the augmentation ideal $I$ into
the group ring $\Z\pi$ and $I^n=I\otimes\dots\otimes I$ is the $n$th
tensor power over $\Z$.

\

We note that all known results on Gromov's macroscopic dimension
\cite{Gr1},\cite{B1},\cite{B2},\cite{B3},\cite{BD} hold true for
$\dim_{MC}$. It could be the case that there is the equality
$\dim_{MC}X=\dim_{mc}X$ for nice metric spaces $X$ such as the
universal coverings of closed manifolds with the lifted metric.

\section{Almost equivariant cohomology}

Let $X$ be a CW complex and let $E_n(X)$ denote the set of its
$n$-dimensional cells. We recall that (co)homology of a CW complex
$X$ with coefficients in an abelian group $G$ are defined by by
means of the cellular chain complex $C_*(X)=\{C_n(X),\partial_n\}$
where $C_n(X)$ is the free abelian group generated by the set
$E_n(X)$. The resulting groups $H_*(X;G)$ and $H^*(X;G)$ do not
depend on the choice of the CW structure on $X$. The proof of this
fact appeals to the singular (co)homology theory and it is a part of
all textbooks on algebraic topology. The same holds true for
(co)homology groups with locally finite coefficients, i.e., for
coefficients in a $\pi$-module $L$ where $\pi=\pi_1(X)$. The chain
complex defining the homology groups $H_*(X;L)$ is $\{C_n(\Wi
X)\otimes_{\pi}L\}$ and the cochain complex defining the cohomology
$H^*(X;L)$ is $Hom_{\pi}(C_n(\Wi X),L)$ where $\Wi X$ is the
universal cover of $X$ with the cellular structure induced from $X$.
The resulting groups $H_*(X;L)$ and $H^*(X;L)$ do not depends on the
CW structure on $X$.

These groups can be interpreted as the equivariant (co)homology:
$$H_*(X;L)=H_*^{lf,\pi}(\Wi X;L)\ \ \ and\ \ \ H^*(X;L)=H_{\pi}^*(\Wi X;L).$$
The last equality is obvious since the equivariant cohomology groups
$H_{\pi}^*(\Wi X;L)$ are defined by equivariant cochains
$Hom_{\pi}(C_n(\Wi X),L\}$. We recall that the equivariant locally
finite homology groups are defined by the complex of infinite {\em
locally finite invariant chains}
$$ C_n^{lf,\pi}(\Wi X;L)=\{\sum_{e\in E_n(\Wi X)} \lambda_ee\mid
\lambda_{ge}=g\lambda_e,\ \lambda_e\in L\}.$$ The local finiteness
condition on a chain requires that for every $x\in\Wi X$ there is a
neighborhood such that the number of $n$-cells $e$ intersecting $U$
for which $\lambda_e\ne 0$ is finite. This condition is satisfied
automatically when $X$ is a locally finite complex. Even in that
case {\em lf} is the part the notation for the equivariant homology
since it was inherited from the singular theory. The following
Proposition implies the equality $H_*(X;L)=H_*^{lf,\pi}(\Wi X;L)$.
\begin{prop}
For every CW complex $X$ with the fundamental group $\pi$ and a
$\pi$-module $L$ the chain complex $\{C_n(\Wi X)\otimes_{\pi}L\}$ is
isomorphic to the chain complex of locally finite equivariant chains
$C_*^{\pi}(\Wi X)$.
\end{prop}

The following definition was given in \cite{Dr} in more general
setting.

\begin{defin} Let $X$ be a CW complex with the universal cover $\Wi X$
and let $L$ be a $\pi$-module. A homomorphism $\phi:C_n(\Wi X)\to L$
is called {\em almost equivariant}, if the set
$$\{\gamma^{-1}\phi(\gamma e)\mid \gamma\in\pi\}\subset L$$ is finite
for every $n$-cell $e$ in $\Wi X$. Let $Hom_{ae}(C_n(\Wi X),L)$ be
the set of all almost equivariant homomorphisms from $C_n(\Wi X)$ to
$L$. Note that this is a group. These groups form a cochain complex
with respect to the co-differential $\delta$ defined as $(\delta
f)(\Delta)=f(\partial\Delta)$ where $\partial: C_n(\Wi X)\to
C_{n-1}(\Wi X)$ are the boundary homomorphisms. The cohomology
groups $H^*_{ae}(\Wi X;L)$ are called the {\em almost equivariant
cohomology} of $\Wi X$ with coefficients in a $\pi$-module $L$.
\end{defin}

One can define singular almost equivariant cohomology by replacing
$n$-cells $e$ in the above definition by  singular simplices
$\sigma:\Delta^n\to \Wi X$. The standard argument show that the
singular version of almost equivariant cohomology coincides with the
cellular. Thus the group $H^*_{ae}(\Wi X;L)$ does not depend on the
choice of a CW complex structure on $X$.

Since every equivariant homomorphism is almost equivariant, there is
a natural transformation
$$pert_X^*:H^*(X;L)=H^*_{\pi}(\Wi X;L)\to H^*_{ae}(\Wi X;L)$$ called
a {\em perturbation homomorphism} from the cohomology of $X$ to the
almost equivariant cohomology. Clearly, for complexes $X$ with
finite fundamental group, $pert_X^*$ is an isomorphism.

Also we note that a proper cellular map $f:X\to Y$ that induces an
isomorphism of the fundamental groups lifts to a proper cellular map
of the universal covering spaces $\bar f:\Wi X\to\Wi Y$. The lifting
$\bar f$ defines a chain homomorphism $\bar f_*:C_n(\Wi X)\to
C_n(\Wi Y)$ and a cochain homomorphism $\bar f^*:Hom_{ae}(C_n(\Wi
Y),L)\to Hom_{ae}(C_n(\Wi X),L)$. The latter defines a homomorphisms
of the almost equivariant cohomology groups
$$\bar f^*_{ae}:H^*_{ae}(\Wi Y;L)\to H^*_{ae}(\Wi X;L).$$

Suppose that $\pi$ acts freely on CW complexes $\Wi X$ and $\Wi Y$
such that the actions preserve the CW complex structures. We call a
cellular map $g:\Wi X\to \Wi Y$ {\em almost equivariant} if the set
$$\bigcup_{\gamma\in\pi}\{\gamma^{-1}g_*(\gamma e)\}\subset C_*(\Wi Y)$$ is finite
for every cell $e$ in $\Wi X$ where $g_*:C_*(\Wi X)\to C_*(\Wi Y)$
the induced chain map.

On a locally finite simplicial complex we consider the geodesic
metric in which every simplex is isometric to the standard.
\begin{prop}\label{induced1}
Let $f:X\to Y$ be a proper almost equivariant cellular map. Then the
induced homomorphism on cochains takes the almost equivariant
cochains to almost equivariant.
\end{prop}
\begin{proof}
Let $\phi:C_n(Y)\to L$ be an almost equivariant cochain. Let $e'$ be
an $n$-cell in $X$. There are finitely many chains $c_1,\dots,c_m\in
C_n(\Wi Y)$ such that $f(\gamma e')\subset \gamma\{c_1,\dots, c_m\}$
for all $\gamma\in\pi$. Then the set
$$\bigcup_{\gamma\in\pi}\{\gamma^{-1}\phi(f(\gamma e'))\}\subset
\bigcup_{\gamma\in\pi}\{\gamma^{-1}\phi(\gamma\{c_1,\dots, c_m\})\}=
\bigcup_{i=1}^m\bigcup_{\gamma\in\pi}\{\gamma^{-1}\phi(\gamma c_i)\}
$$ is finite.
\end{proof}

For every CW-complex $X$ we consider the product CW-complex
structure on $X\times[0,1]$ with the standard cellular structure on
$[0,1]$.

Proposition~\ref{induced1} and the standard facts about cellular
chain complexes imply the following.
\begin{prop}\label{induced2}
Let $X$ and $Y$ be  complexes with free cellular actions of a group
$\pi$.

(A) Then every almost equivariant cellular map $f:X\to Y$ induces an
homomorphism of the almost equivariant cohomology groups
$$
f^*:H^*_{ae}(Y;L)\to H^*_{ae}(X;L).$$

(B) If two almost equivariant maps $f_1, f_2: X\to Y$ are homotopic
by means of a cellular almost equivariant homotopy
$H:X\times[0,1]\to Y$, then they induce the same homomorphism of the
almost equivariant cohomology groups, $f^*_1=f^*_2$.
\end{prop}

Similarly one can define the almost equivariant homology groups on a
CW complex by considering infinite locally finite almost equivariant
chains.  Let $X$ be a complex with the fundamental group $\pi$ and
the universal cover $\Wi X$. We call an infinite chain $\sum_{e\in
E_n(\Wi X)} \lambda_ee$ {\em almost equivariant} if the set
$\{\gamma^{-1}\lambda_{\gamma e}\mid\gamma\in\pi\}\subset L$ is
finite for every cell $e$. As we already have mentioned, the complex
of equivariant locally finite chains defines equivariant locally
finite homology $H^{lf,\pi}_*(\Wi X;L)$. The homology defined by the
almost equivariant locally finite chain we call  {\em the almost
equivariant locally finite homology}. We denote them as
$H^{lf,ae}_*(\Wi X;L)$.  We note that like in the case of cohomology
this definition can be carried out for the singular homology and it
gives the same groups. In particular the groups $H^{lf,ae}_*(\Wi
X;L)$ do not depend on the choice of a CW complex structure on $X$

As in the case of cohomology for any complex $K$ there is a
perturbation homomorphism
$$ pert_*^K:H_*(K;L)=H^{lf,\pi}_*(\Wi K;L)\to H_*^{lf,ae}(\Wi K;L).$$
Also, there is an analog of Proposition~\ref{induced2} for the
almost equivariant locally finite homology.

\begin{prop}\label{h-induced}
Let $X$ and $Y$ be  complexes with free cellular actions of a group
$\pi$.

(A) Then every almost equivariant cellular map $f:X\to Y$ induces a
homomorphism of the almost equivariant homology groups
$$
f_*:H_*^{lf,ae}(X;L)\to H_*^{lf,ae}(Y;L).$$

(B) If two almost equivariant maps $f_1, f_2: X\to Y$ are homotopic
by means of a cellular almost equivariant homotopy, then they induce
the same homomorphism of the almost equivariant cohomology groups,
$(f_1)_*=(f_2)_*$.
\end{prop}

Let $X_i$, $i=1,2$ be complexes with free action of $\pi_i$ and
$L_i$ be $\pi_i$-modules. The tensor product on locally finite
chains
$$
C_k^{lf}(X_1,L_1)\otimes C_l^{lf}(X_2,L_2)\to C^{lf}_{k+l}(X_1\times
X_2,L_1\otimes L_2)$$ defined by the formula
$$
\sum m_{\sigma}\sigma\otimes\sum n_{\kappa}\kappa \to\sum
(m_{\sigma}\otimes n_{\kappa})(\sigma\times\kappa)$$ takes the
product of almost equivariant chains to almost equivariant. Since
the above tensor product defines homomorphisms $\phi_*$ and
$\phi_*^{ae}$ for both the equivariant and the almost equivariant
homology, we obtain the following:
\begin{prop}\label{tensor}
For any complexes $M$ and $N$ with universal coverings $\Wi M$ and
$\Wi N$, for any $\pi_(M)$ and $\pi_1(N)$ modules $L_1$ and $L_2$,
and for any $k$ and $l$, there is a commutative diagram:
$$
\begin{CD}
H_k(M;L_1)\otimes H_l(N;L_2) @>\phi_*>> H_{k+l}(M\times N;L_1\otimes
L_2)\\
@V{pert_*^M\otimes pert_*^N}VV @V{pert_*^{M\times N}}VV\\
H_k^{lf,ae}(\Wi M;L_1)\otimes H_l^{lf,ae}(\Wi N;L_2) @>\phi_*^{ae}>>
H_{k+l}^{lf.ae}(\Wi M\times \Wi N;L_1\otimes
L_2).\\
\end{CD}
$$
\end{prop}

\

Let $M$ be an oriented $n$-dimensional PL manifold with a fixed
triangulation. Denote by $M^*$ the dual complex. There is a
bijection between $k$-simplices $e$ and the dual $(n-k)$-cells $e^*$
which defines the Poincare duality isomorphism. This bijection
extends to a similar bijection on the universal cover $\Wi M$. Let
$\pi=\pi_1(M)$. For any $\pi$-module $L$ the Poincare duality on $M$
with coefficients in $L$ is given by the cochain-chain level by
isomorphisms
$$
Hom_{\pi}(C_k(\Wi M^*),L) \stackrel{PD_k}\longrightarrow
C_{n-k}^{lf,\pi}(\Wi M;L)$$ where $PD_k$ takes a cochain
$\phi:C_k(\Wi M^*)\to L$ to the following chain $\sum_{e\in
E_{n-k}(\Wi M)}\phi(e^*)e$. The family $PD_*$ is a chain isomorphism
which is also known as the cap product
$$
PD_k(\phi)=\phi\cap[\Wi M]$$ with the fundamental class $[\Wi M]\in
C_n^{lf,\pi}(\Wi M)$, where $[\Wi M]=\sum_{e\in E_n(\Wi M)} e$. We
note that the homomorphisms $PD_k$ and $PD_k^{-1}$ extend to the
almost equivariant chains and cochains:
$$
Hom_{ae}(C_k(\Wi M^*),L) \stackrel{PD_k}\longrightarrow
C_{n-k}^{lf,ae}(\Wi M;L).$$ Thus, the homomorphisms $PD_*$ define
the Poincare duality isomorphisms $PD_{ae}$ between the almost
equivariant cohomology and homology. We summarize this in the
following
\begin{prop}\label{PD}
For any closed oriented $n$-manifold $M$ and any $\pi_1(M)$-module
$L$ the Poincare duality forms the following commutative diagram:
$$
\begin{CD}
H^k(M;L) @ >pert^*_M>> H^k_{ae}(\Wi M;L)\\
@V{-\cap [M]}VV @V{PD_{ae}}VV\\
H_{n-k}(M;L) @>pert_*^M>> H_{n-k}^{lf,ae}(\Wi M;L).\\
\end{CD}
$$
\end{prop}

We note that the operation of the cap product for equivariant
homology cohomology automatically extends on the chain-cochain level
to the cap product on the almost equivariant homology and
cohomology. Then the Poincare Duality isomorphism $PD_{ae}$ for $\Wi
M$ can be described as the cap product with the homology class
$pert_*^M([M])$.

\section{Obstruction to the inequality $\dim_{MC}\Wi M^n<n$}

Let $\pi$ be a finitely presented group. Then the classifying space
$B\pi=K(\pi,1)$ can be taken to be a locally finite complex. We fix
a geodesic metric on $B\pi$. Let $p_{\pi}:E\pi\to B\pi$ denote the
universal covering. We consider the induced CW complex structure and
induced geodesic metric on $E\pi$.

\begin{prop}\label{sweep}
Let $K$ be a finite complex with the universal cover $p:\Wi K\to K$
supplied by a geodesic metric induced from $K$. Let $f:K\to B\pi$ be
a cellular Lipschitz map classifying $p$. Suppose that $\dim_{MC}\Wi
K< n$. Then for every lift $\Wi f:\Wi K\to E\pi$ of $fp$ there is a
Lipschitz cellular homotopy $H:\Wi K\times I\to E\pi$ of $\Wi f$ to
a map $g:\Wi K\to E\pi^{(n-1)}$ where $E\pi^{(n-1)}$ denotes the
$(n-1)$-skeleton of $E\pi$.
\end{prop}
\begin{proof}
We may assume that $B\pi$ is a locally finite complex. Moreover,
since the group $\pi=\pi_1(K)$ is finitely presented, we may assume
that the 2-skeleton $B\pi^{(2)}$ is a finite complex. Let $\phi:\Wi
K\to N$ be a uniformly cobounded Lipschitz map to an
$(n-1)$-dimensional simplicial complex. We may assume that $\phi$ is
cellular and surjective on the cell level, that is an every cell in
$N$ has nonempty intersection with the image $\phi(\Wi K)$. Also we
may assume that there is $C>0$ such that $diam(\phi^{-1}(\Delta))<C$
of all simplices $\Delta$ in $N$. We construct a Lipschitz map
$q:N\to E\pi^{(n-1)}$ by induction on dimension of the skeleton of
$N$. For every $v\in N^{(0)}$ we define $q(v)$ be a closest vertex
in $E\pi$ to the set $\Wi f(\phi^{-1}(v))$. Then for every edge
$[v,v']$ in $N$ we define $q([v,v'])$ to be a shortest path in the
Cayley graph $E\pi^{(1)}$ taken with the graph metric from $q(v)$ to
$q(v')$. Since the distance between $q(v)$ and $q(v')$ for all edges
$[v,v']$ is uniformly bounded, there is an upper bound on the number
of translational isometries of such paths. Then for every 2-simplex
$\sigma=[v_0,v_1,v_2]$ is $N$ we take a filing of
$q(\partial\sigma)$ in $E\pi^{(2)}$ that uses a minimal number of
2-cells and so on.

Let $\scr H_k$ denote the set of all simplicial imbeddings
$h:\Delta^k\to N$ of the standard $k$-simplex. After the step number
$k<n$ we will get a Lipschitz map $q:N^{(k)}\to E\pi^{(k)}$ such
that the family of maps $\{p_{\pi}qh:\partial\Delta^{k+1}\to
B\pi\}_{h\in\scr H_k}$ is finite. Then we can construct a Lipschitz
extension $q:N^{(k+1)}\to E\pi^{(k+1)}$ using minimal fillings in
$B\pi$. As the result we obtain a Lipschitz map $q:N\to E\pi^{(n)}$
with $p_{\pi}q(N)$ compact. We note that the composition $q\phi$ is
on bounded distance from $\Wi f$. Clearly, the maps $\Wi f$ and
$q\phi$ are homotopic as maps to a contractible space. Since $q\phi$
and $\Wi f$ are on bounded distance with compact projections
$p_{\pi}\Wi f(\Wi X)$ and $p_{\pi}q\phi(\Wi X)$, there is a
uniformly bounded homotopy between them. Then we can turn that
homotopy to a cellular Lipschitz map.
\end{proof}

Let $A$ be a subset of a CW complex $X$. The star neighborhood
$St(A)$ of $A$ is the closure of the union of all cells in $X$ that
have a nonempty intersection with $A$. Note that $St(A)$ is a
subcomplex of $X$.
\begin{prop}\label{uniform=almost} Let $X$ and $Y$ be as above with $Y$ locally finite.
Then a cellular Lipschitz homotopy $\Phi:X\times[0,1]\to Y$ of an
almost equivariant map is almost equivariant.
\end{prop}
\begin{proof}
Since $\Phi|_{X\times\{0\}}$ is almost equivariant, for every cell
$e\subset X$ the union
$$\bigcup_{\gamma\in\pi}\Phi(e\times\{0\})\subset
\sigma_1\cup\dots\cup\sigma_k.$$ Then
$$\bigcup_{\gamma\in\pi}\Phi(e\times(0,1))\subset
St^m(K)$$ where $K$ is the closure of
$\sigma_1\cup\dots\cup\sigma_k$. The existence of $m$ follows from
the fact that $\Phi$ is Lipschitz. The local finiteness of $Y$ and
finiteness of $K$ imply that the $m$-times iterated star
neighborhood $St^m(K)$ of $K$ is a finite subcomplex of $Y$.
Clearly, the coefficients of the cells in $\Phi_*(e\times(0,1))\in
C_*(Y)$ are bounded.
\end{proof}

Here we recall some basic facts of the elementary obstruction
theory. Let $f:X\to Y$ be a cellular map that induces an isomorphism
of the fundamental groups. We want to deform the map $f$ to a map to
the $(n-1)$-skeleton $Y^{(n-1)}$. For that we consider the extension
problem $$X\supset X^{(n-1)} \stackrel{f}{\to} Y^{(n-1)},$$ i.e.,
the problem to extend $f:X^{(n-1)}\to Y^{(n-1)}$ continuously to a
map $\bar f:X\to Y^{(n-1)}$. The primary obstruction for this
problem $o_f$ is the obstruction to extend $f$ to the $n$-skeleton.
It lies in the cohomology group $H^n(X;L)$ where
$L=\pi_{n-1}(Y^{(n-1)})$ is the $(n-1)$-dimensional homotopy group
considered as a $\pi$-module for $\pi=\pi_1(Y)=\pi_1(X)$. The
obstruction theory says that a map $g:X\to Y^{(n-1)}$ that agrees
with $f$ on the $(n-2)$-skeleton $X^{(n-2)}$ exits if and only if
$o_f=0$. The primary obstruction is natural: If $g:Z\to X$ is a
cellular map, then $o_{gf}=g^*(o_f)$. In particular, in our case
$o_f=f^*(o_1)$ where $o_1\in H^n(Y;L)$ is the primary obstruction to
the retraction of $Y$ to the $(n-1)$-skeleton.

\begin{defin}
Let $g:Y^{(n-1)}\to Z$ be a Lipschitz map of the $(n-1)$-skeleton of
an $n$-dimensional  complex to a metric space. We call the problem
to extend $g$ to a Lipschitz map $\bar g:Y\to Z$  a {\em Lipschitz
extension problem }.
\end{defin}

\begin{defin}\label{obst}
Let $X$ be a finite $n$-complex, $n\ge 3$, with $\pi_1(X)=\pi$
and $\Wi f:\Wi X\to \Wi Y$ be a lift of a cellular map $f:X\to Y$
that induces an isomorphism of the fundamental groups. We define an
element $o_{\Wi f}\in H^n_{ae}(\Wi X;\pi_{n-1}(Y^{(n-1)}))$ as the
class of the cocycle
$$C_{\Wi f}:C_n(\Wi X)\to\pi_{n-1}(\Wi Y^{(n-1)})=\pi_{n-1}(Y^{(n-1)})$$
defined by the formula $C_{\Wi f}(e)=[\Wi f\circ\phi_e]$ where
$\phi_e:S^{n-1}=\partial B^n\to X$ is the attaching map of an
$n$-cell $e$. Since the map $\Wi f$ is $\pi$-equivariant, the
cocycle $C_{\Wi f}$ is a $\pi$-equivariant. Thus, it defines an
element $\kappa_{f}\in H^n_{\pi}(\Wi X;\pi_{n-1}(Y^{(n-1)}))$ of the
equivariant cohomology and $o_{\Wi f}=pert^*_X(\kappa_{f})$.
\end{defin}

We consider an arbitrary geodesic metric on a locally finite complex
and the induced metric on its universal cover.
\begin{prop}\label{obstructiontheory}
Let $\Wi f:\Wi X\to\Wi Y$ be a lift of a Lipschitz cellular map
$f:X\to Y$ of a finite $n$-dimensional complex to a locally finite
that induces an isomorphism of the fundamental groups.  Then the
above cohomology class $o_{\Wi f}\in H^n_{ae}(\Wi
X;\pi_{n-1}(Y^{(n-1)}))$ is the primary obstruction for the
following Lipschitz extension problem
$$\Wi X\supset \Wi X^{(n-1)}\stackrel{\Wi f|}{\to} \Wi Y^{(n-1)}.$$
Thus, $o_{\Wi f}=0$ if and only if there is a Lipschitz map $\bar
g:\Wi X\to \Wi Y^{(n-1)}$ which agrees with $\Wi f$ onto $\Wi
X^{(n-2)}$.
\end{prop}
\begin{proof}
The proof goes along the lines of a similar statement from the
classical obstruction theory. Let $C_{\Wi f}=\delta\Psi$ where
$\Psi:C_{n-1}(\Wi X)\to \pi_{n-1}(\Wi Y^{(n-1)})$ be an almost
equivariant homomorphism. For each $(n-1)$-cell $e$ of $X$ we fix a
section $\Wi e\subset\Wi X$, an $(n-1)$-cell in $\Wi X$. Then the
set $\{\gamma^{-1}\Psi(\gamma\Wi e)\mid
\gamma\in\pi\}=\{\gamma_i^{-1}\Psi(\gamma_i\Wi e)\}_{i=1}^m$ is
finite. Like in the classical obstruction theory we define a map
$g_i:\gamma_i\Wi e\to Y^{(n-1)}$, $i=1,\dots, m$ on cells
$\gamma_i\Wi e$ such that $g_i$ agrees with $\Wi f$ outside a small
$(n-1)$-ball $B_i\subset\gamma_i\Wi e$ and  the difference of $\Wi
f$ and $g_i$ restricted to $B_i$ defines a map $d_{\Wi
f,g_i}:S^{n-1}\to Y^{(n-1)}$ that represents the class
$-\Psi(\gamma_i\Wi e)$. For a general lift $\gamma\Wi e$ of $e$ we
define a map $g:\gamma\Wi e\to Y^{(n-1)}$ as follows. Let $i$ be
such that $\gamma^{-1}\Psi(\gamma\Wi
e)=\gamma_i^{-1}\Psi(\gamma_i\Wi e)$. We define
$g=\gamma\gamma_i^{-1} g_i\gamma_i\gamma^{-1}$. Thus, we define
$g:\Wi X^{(n-1)}\to\Wi Y^{(n-1)}$ in such a way that the difference
map $d_{\Wi f,g}:S^{n-1}\to Y^{(n-1)}$ on the cell $\gamma\Wi e$,
$\gamma\in\pi$, represents the element
$-(\gamma_i\gamma)^{-1}\Psi(\gamma_i\Wi e)=-\Psi(\gamma\Wi e)$. Then
the elementary obstruction theory implies that for every $n$-cell
$\sigma'\subset\Wi X$ there is an extension $\bar
g_{\sigma'}:\overline{\sigma'}\to\Wi Y^{(n-1)}$ of
$g|_{\partial\sigma'}$. For every $n$-cell $\sigma\subset X$ we fix
a lift $\Wi\sigma$. Consider the set of maps
$$\bigcup_{\gamma\in\pi}\{\gamma^{-1}g\gamma|_{\partial\Wi\sigma}:\partial\Wi
\sigma\to Y^{(n-1)}\}\subset C(\partial\Wi\sigma,Y^{(n-1)}).$$ By
the construction of $g$, this set is finite. We fix an extension
$\bar g_{\sigma'}$ for each element of this set and define the
extension $\bar g:X\to Y^{(n-1)}$ of $g$ by translations by $\pi$.
We may assume that all maps $g_i$ and $\bar g_{\sigma'}$ in the
above construction are Lipschitz with the same Lipschitz constant.
We can do that since there are finitely many maps there.

In the other direction, if there is a Lipschitz map $\bar g:\Wi X\to
Y^{(n-1)}$ that coincides with $\Wi f$ on the $(n-2)$-dimensional
skeleton, then the difference cochain $d_{\Wi f,\bar g}$ is almost
equivariant. Indeed, for any $\lambda>0$ there are finitely many
homotopy classes in $\pi_{n-1}(Y)$ can be realized by
$\lambda$-Lipschitz maps. Then the formula $\delta d_{\Wi f,\bar g}=
C_{\bar g}-C_{\Wi f}$ and the fact that $o_{\bar g}=0$ imply that
$o_{\Wi f}=0$.
\end{proof}

Let $[e]\in \pi_{n-1}(\Wi Y^{(n-1)})$ denote the element of the
homotopy group defined by the attaching map of an $n$-cell $e$. Then
the homomorphism $C_{\Wi 1}: C_n(\Wi Y)\to\pi_{n-1}(\Wi Y^{(n-1)})$
defined as $C_{\Wi 1}(e)=[e]$ is an equivariant cocycle with the
cohomology class $o_{\Wi 1}\in H^n_{ae}(\Wi
Y;\pi_{n-1}(Y^{(n-1)}))$.

\begin{prop}\label{obstruction}
(1) The cohomology class $o_{\Wi f}$ from the above Proposition is
the image under $\Wi f^*$ of the class $o_{\Wi 1}\in H^n_{ae}(\Wi
Y;\pi_{n-1}(Y^{(n-1)}))$.

(2) The class $o_{\Wi 1}$ comes under the homomorphism
$pert_{\pi}^*$ from the primary obstruction $\kappa_1\in H^n(Y;
\pi_{n-1}(Y^{(n-1)}))$ to retract $Y$ to the $(n-1)$-dimensional
skeleton.
\end{prop}
\begin{proof}
The first part is the naturality of obstructions for Lipschitz
extension problems with respect to Lipschitz maps. Like in the case
of classical obstruction theory, it follows from the definition.

The second part follows from definition (see Definition~\ref{obst}).
\end{proof}

\begin{prop}\label{ball-cobound}
Let $f:X\to B\pi$ be a Lipschitz map of a finite complex that
induces an isomorphism of the fundamental groups and let $\Wi f:\Wi
X\to E\pi$ be its lift to the universal coverings. Then for every
$R>0$ the family of preimages of $R$-balls $\{\Wi
f^{-1}(B_{R}(y))\}_{y\in E\pi}$ is uniformly bounded where the
metrics on $\Wi X$ and $E\pi$ are induced from geodesic metrics on
$X$ and $B\pi$.
\end{prop}
\begin{proof}
Since $\Wi F\Wi X\to\Wi f(\Wi X)$ is a quasi-isometry for the subset
metric on $\Wi f(\Wi X)$, the family $\{\Wi f^{-1}(B_{R}(y))\}_{y\in
\Wi f(\Wi X)}$ is uniformly cobounded. Then, clearly, the family
$\{\Wi f^{-1}(B_{R}(y))\}_{y\in E\pi}$ is uniformly bounded.
\end{proof}

\begin{thm}\label{obstr-dim}
Let $X$ be a finite $n$-complex with $\pi_1(X)=\pi$ and let $f:X\to
B\pi$ be a Lipschitz map that induces an isomorphism of the
fundamental groups. Then $\dim_{MC}\Wi X<n$, $n\ge 3$, if and only
if the above obstruction is trivial, $o_{\Wi f}=0$.
\end{thm}
\begin{proof}
If $\dim_{MC}\Wi X<n$, then by Proposition~\ref{sweep} there is a
Lipschitz cellular homotopy of $\Wi f:\Wi X\to E\pi$ to a map $g:\Wi
X\to E\pi^{(n-1)}$. By Proposition~\ref{uniform=almost}, the map $g$
is almost equivariant. Then by Proposition~\ref{induced2}, $o_{\Wi
f}=\Wi f^*(o_1)=g^*i^*(o_1)=0$ where $i:E\pi^{(n-1)}\subset E\pi$ is
the inclusion.

We assume that $B\pi$ is a locally finite simplicial complex. If
$o_{\Wi f}=0$, then by Proposition~\ref{obstructiontheory} there is
a $\lambda$-Lipschitz map $g:\Wi X\to E\pi^{(n-1)}$ for some
$\lambda$ which agrees with $\Wi f$ on the $(n-2)$-skeleton
$X^{(n-2)}$. We may assume that $f$ is also $\lambda$-Lipschitz.
Additionally, we assume that the diameter of each cell in $X$ is
less than 1. We show that the map $g$ is uniformly cobounded. In
view of Proposition~\ref{ball-cobound} it suffices to show that
$g^{-1}(y)\subset \Wi f^{-1}(B_{3\lambda}(y))$ for all $y\in E\pi$.
Let $x\in g^{-1}(y)$ and let $x\in e$ where $e$ is a cell in $\Wi
X$. Since $X$ is $n$-dimensional, there is a point $v\in
X^{(n-2)}\cap St(\bar e)$. Then $d_{\Wi X}(x,v)<2$ and
$d(g(x),g(v))<2\lambda$ where $d$ is the metric on $E\pi$ induced
from a proper geodesic metric on $B\pi$. Thus, $d_{\Wi Y}(y,\Wi
f(v))<2\lambda$. Since the map $\Wi f$ is $\lambda$-Lipschitz, by
the triangle inequality and the fact that $\Wi f(v)=g(v)$,
$$d(y,\Wi f(x))<d(y,g(v))+d(\Wi f(v),\Wi f(x))< 3\lambda,$$ i.e.,
$\Wi f(x)\in B_{3\lambda}(y)$. Therefore, $x\in\Wi
f^{-1}(B_{3\lambda}(y))$.
\end{proof}

\section{Homology of groups}

Let $I$ denote the augmentation ideal of the group ring $\Z\pi$. We
recall that the {\em Berstein-\v Svarc class} $\beta=\beta_{\pi}\in
H^1(\pi;I)$ is the first obstruction to the lift of $B\pi$ to $E\pi$
(see \cite{Sv} and \cite{DR}). The following is called the
Universality Theorem and it is stated  without proof in \cite{Sv}. A
proof can be found in \cite{DR}.
\begin{thm}\label{univ}
For every $\pi$-module $L$ and every cohomology class $\alpha\in
H^k(\pi;L)$ there is a $\pi$-homomorphism $I^k\to L$ that takes
$\beta^k$ to $\alpha$.
\end{thm}
Here $\beta^k=\beta\smile\dots\smile\beta$ is the $k$ times cup
product and $I^k=I\otimes\dots\otimes I$ is the $k$ times tensor
product over $\Z$. We recall that the cup product $x\smile  y$ of
classes $x\in H^*(X;A)$ and $y\in H^*(X;B)$ is defined for any
modules $A$ and $B$ with $x\smile y\in H^*(X;A\otimes B)$ \cite{Br}.

Let $f:M\to B\pi$ be a map that induces an isomorphism of the
fundamental groups. The image $f^*(\beta_{\pi})\in H^1(M;I)$ of the
Berstein-\v Svarc class of $\pi$ is denoted by $\beta_M$ and is
called Berstein-\v Svarc class of $M$.

We will use the notations $H^*_{ae}(\pi;L)$, $H_*^{lf,ae}(\pi;L)$,
$pert^*_{\pi}$ and $pert_*^{\pi}$ for $H^*_{ae}(E\pi;L)$,
$H_*^{lf,ae}(E\pi;L)$, $pert^*_{B\pi}$ and $pert_*^{B\pi}$
respectively. Also we will use the notation $H_*(\pi)$ for
$H_*(\pi;\Z)$.

\begin{thm}\label{small}
For a closed oriented $n$-manifold $M$ the following are equivalent:

1. $\dim_{MC}\Wi M<n$;

2. $f_*([M])\in ker(pert_*^{\pi})$ where $f:M\to B\pi$ is the map
classifying the universal covering $\Wi M$ of $M$.

3. $(\beta_M)^n\in ker(pert^*_M)$.
\end{thm}
\begin{proof}
1. $\Rightarrow$ 2. Let $f:M\to B\pi$ be a cellular Lipschitz map
classifying the universal cover $\Wi M$ of $M$ and let $\Wi f:\Wi
M\to E\pi$ be a lift. If $\dim_{MC}\Wi M<n$, then by
Proposition~\ref{sweep} there is a Lipschitz cellular homotopy of
$\bar f:\Wi X\to E\pi$ to a map $g:\Wi X\to E\pi^{(n-1)}$. By
Proposition~\ref{uniform=almost}, it is almost equivariant. Then by
Proposition~\ref{induced2} it follows that $\Wi
f_*(pert_*^M([M]))=0$. Therefore, $pert_*^{\pi}(f_*([M]))=0$ and
hence, $f_*([M])\in ker(pert_*^{\pi})$.

2. $\Rightarrow $ 3. If $f_*([M])\in ker(pert_*^{\pi})$, then
$pert_*^{\pi}(f_*([M])\cap\beta^n)=0$. Since we may assume that the
restriction of $f$ to the 1-skeleton of $M$ is a homeomorphism of
1-skeletons, the commutative diagram
$$
\begin{CD}
H_0^{lf,ae}(\Wi M;I^n) @>\bar f_*>> H_0^{lf,ae}(E\pi;I^n)\\
@Apert_*^MAA @Apert_*^{\pi}AA\\
H_0(M;I^n) @>f_*>> H_0(B\pi;I^n)\\
\end{CD}
$$
has isomorphisms for horizontal arrows. Therefore, $pert_*^M([M]\cap
(f^*\beta)^n)=0$. Thus, $pert_*^M([M])\cap
pert^*_M((f^*\beta)^n)=0$. By the Poincare Duality,
$pert^*_M((\beta_M)^n)=0$.

3. $\Rightarrow$ 1. By Proposition~\ref{obstruction},  $o_{\Wi
f}=pert^*_M(f^*(\kappa_1)$. By the Universality Theorem there is a
coefficient homomorphism $\psi:I^n\to\pi_{n-1}(B\pi^{(n-1)})=L$ such
that the induced homomorphism of the $n$th cohomology groups takes
$\beta^n$ to $\kappa_1$. Therefore, $\psi$ induces the commutative
diagram
$$
\begin{CD}
H^n_{ae}(\Wi M;I^n) @>\psi_*>> H^n_{ae}(\Wi M;L)\\
@ApertAA @ApertAA\\
H^n(M;I^n) @>\psi_*>> H^n(M;L)\\
\end{CD}
$$
where $$o_{\Wi
f}=pert^*_M(f^*(\kappa_1))=pert^*_M(\psi_*(\beta_M)^n))=\psi_*(pert^*_M((\beta_M)^n)=0.$$
Then by Theorem~\ref{obstr-dim} $\dim_{MC}\Wi M<n$.
\end{proof}

We note that the subset of $n$-homology classes of $H_n(\pi)$ which
can be realized by an $n$-manifolds forms a subgroup. We denote this
subgroup by $RH_n(\pi)$ and call it the {\em representable
$n$-homology group}. Using the surgery one can show that for $n\ge
4$ a realization $f:M\to B\pi$ of a given class from $RH_n(\pi)$ can
be taken such that $f$ induces an isomorphism of the fundamental
groups.

\begin{defin} We define the group of small macroscopic dimension
classes as $H^{sm}_n(\pi)=ker(pert_*^{\pi})\cap RH_n(\pi)\subset
H_n(\pi)$.

\end{defin}

\begin{cor}\label{smallhom}
For a closed orientable $n$-manifold the following are equivalent:

(1) $\dim_{MC}\Wi M<n$;

(2) $f_*([M])\in H^{sm}_n(\pi)$.
\end{cor}

We conclude this section with the following observation about
homologies of a group.

\begin{thm}\label{shift}
For every group $\pi$ and any $n>0$,
$$H_n(\pi)=H_{n-1}(\pi;I)=H_{n-2}(\pi;I^2)=\dots=H_1(\pi;I^{n-1})$$ and
$$ H_n(\pi)=ker\{H_0(\pi;I^n)\to H_0(\pi;I^{n-1}\otimes\Z\pi)\}$$ or
to state the same differently,
$$H_n(\pi)=ker\{I^n\otimes_{\pi}\Z
\stackrel{(1\otimes i)\otimes
1}\longrightarrow(I^{n-1}\otimes\Z\pi)\otimes_{\pi}\Z\}$$ where
$i:I\to\Z\pi$ is the imbedding.
\end{thm}
\begin{proof}
The first chain of equalities follow from the homology long exact
sequence defined by the short exact sequence of coefficients
$$ 0\to I^k\to I^{k-1}\otimes\Z\pi\to I^{k-1}\to 0$$
and the fact that $H_i(B\pi;I^{k-1}\otimes\Z\pi)=0$ for $i>0$. The
latter is due to the facts that the reduced homologies of a group
with coefficients in a projective module are trivial \cite{Br} and
the modules $I^{k-1}\otimes\Z\pi$ are projective \cite{DR}. Here we
use the convention $I^0=\Z$.

The second equality follows from the facts that
$H_n(\pi)=H_1(\pi;I^{n-1})$ and the remainder of the coefficients
exact sequence is
$$0\to H_1(\pi;I^{n-1})\to H_0(\pi;I^n) \to
H_0(\pi; I^{n-1}\otimes\Z\pi)\to H_0(\pi;I^{n-1}).$$ The last
equality follows from definition of 0-dimensional homology:
$$H_0(\pi;L)=coinv(L)=L\otimes_{\pi}\Z$$ for every $\pi$-module $L$.
\end{proof}

\section{Uniformly finite homology}
Let $X$ be a uniform simplicial complex. Block and Weinberger
introduced the uniformly finite homology groups $H^{uf}_n(X;\Z)$ as
the homology groups of  the chain complex of bounded infinite chains
$$C^{uf}_k=\{\sum_{\sigma\in E_k(X)} n_{\sigma}\sigma,\ \  k\in\N,\ \  n_{\sigma}\in\Z,\ \
 \max\{|n_{\sigma}|\}<\infty\}$$ where
$\sigma$ runs over all $k$-simplices of $X$. They defined the
uniformly finite homology for general metric spaces and we refer to
\cite{BW} for the precise definition in the general case.

REMARK 1. Let $\Wi K$ be the universal covering of a finite
simplicial complex $K$ with the fundamental group $\pi$. We assume
that $\Wi K$ is given the metric lifted from $K$.  When $L=\Z$  is a
trivial $\pi$-module, the almost equivariant locally finite homology
groups $H_*^{lf,ae}(\Wi K;L)$ coincide with the uniformly finite
homology $H^{uf}_*(\Wi K;\Z)$.

 The following theorem is due to Block and Weinberger \cite{BW}.
\begin{thm}\label{BW}
For a finite complex $K$, $H_0^{uf}(\Wi K;\Z)=0$ if and only if
$\pi_1(K)$ is not amenable.
\end{thm}

Here is our main result.

\begin{thm}\label{main}
There is a closed rationally essential $n$-manifold $M$, $n\ge 5$,
with the fundamental group $\pi_1(M)=\Z^n\times F_2$ such that
$\dim_{MC}\Wi M<n$.
\end{thm}
\begin{proof}
We note that $B\pi=T^n\times(S^1\vee S^1)$ for $\pi=\pi_1(M)$ where
$T^n$ is the $n$-torus. Consider the natural inclusion of $T^n$ into
$B\pi$. Then the image of the fundamental class  $[T^n]$ in
$H_n(B\pi)$ is $[T^n]\otimes 1$ where $1\in H_0(S^1\vee S^1)$. By
Remark 1 and Theorem~\ref{BW}, $pert_*^{F_2}(1)=0$. Therefore, in
view of Proposition~\ref{tensor}
$$pert_*^{\pi}([T^n]\otimes 1)= pert^{\Z^n}_*([T^n])\otimes
pert_*^{F_2}(1)=0.$$ By a surgery in dimension 1 and 2 performed on
the torus $T^n$ we can obtain a manifold $M$ together with a map
$f:M\to B\pi$ inducing isomorphism of the fundamental groups and
such that $f([M])=[T^n]\otimes 1$. By Theorem~\ref{obstr-dim}
$\dim_{MC}\Wi M<n$.
\end{proof}
REMARK 2. The free group $F_2$ in the Theorem~\ref{main} can be
replaced by any non-amenable group.


\begin{thebibliography}{[Gr1]}

\bibitem[Ba]{Ba}
A. Bartels, {\em Squeezing and higher algebraic K-theory}, K-theory
vol 28 (2003), 19-37.

\bibitem[BR]{BR} A. Bartels, D. Rosenthal, {\em On the K -theory of groups with finite
asymptotic dimension}, J. Reine Angew. Math. 612 (2007) 35–57.

\bibitem[BW]{BW}
J. Block, S. Weinberger
\newblock {Aperiodic tilings, positive scalar curvature and amenability
of spaces}, J. Amer. Math. Soc.
{\bf 5} no. 4 (1992), 907-921.



\bibitem[B1]{B1}
    D. Bolotov
\newblock {\em Macroscopic dimension of $3$-Manifolds},
\newblock {Math. Physics, Analysis and Geometry
 {\bf 6} (2003), 291 - 299}.

\bibitem[B2]{B2}
    D. Bolotov
\newblock {\em Gromov's  macroscopic dimension conjecture },
\newblock {Algebraic and Geometric Topology
 {\bf 6} (2006), 1669 - 1676}.

\bibitem[B3]{B3}
D. Bolotov, {\em Macroscopic dimension of certain PSC-manifolds},
AGT {\bf 9}, (2009) 31-27.

\bibitem[BD]{BD}
D. Bolotov, A. Dranishnikov
\newblock  {\em On Gromov's scalar curvature conjecture},
Proc. AMS, {\bf 138} no. 4 (2010), 1517-1524,  \newblock {Preprint
arXiv:0901.4503v1 [math.GT]} 2009.

\bibitem[Br]{Br}
K. Brown
\newblock {Cohomology of groups},
\emph{Graduate Texts in Mathematics}, \textbf{87}
Springer, New York Heidelberg Berlin, 1994.

\bibitem[BH]{BH}
M. Brunnbauer, B. Hanke
\newblock {\em Large and small group homology},
\newblock {Preprint, ArXiv:0902.0869v} 2009.

\bibitem[CG]{CG} G. Carlsson and B. Goldfarb {\em The integral K-theoretic
Novikov conjecture for groups with finite asymptotic dimension}
Invent. Math. (2004) vol 157 No 2, 405--418.

\bibitem[Dr]{Dr}
A. Dranishnikov, {\em Macroscopic dimension and essential
manifolds}. Proceedings of the conference dedicated to 75th
anniversary of Steklov Mathematical Institute held in Moscow in
2009, to appear.


\bibitem[DFW]{DFW}
A. Dranishnikov, S. Ferry, and  S. Weinberger, {\em An Etale
approach to the Novikov conjecture}, Pure Appl. Math. 61 (2008), no.
2, 139-155.

\bibitem[DR]{DR}
A. Dranishnikov, Yu. Rudyak, {\em On the Berstein-\v Svarc Theorem
in dimension 2}. Math. Proc. Cambridge Phil. Soc. {\bf 146} (2009),
407-413.




\bibitem [Gr1]{Gr1}
    M. Gromov,
\newblock  {\em Positive curvature, macroscopic dimension, spectral
gaps and higher signatures}, Functional analysis on the eve of the
21st century. Vol II, Birhauser, Boston, MA, 1996.

\bibitem [Gr2]{Gr2}
    M. Gromov,
\newblock {\em Filling Riemannian manifolds}
\newblock {J. Differential Geometry}
{\bf 18} (1983), 1 - 147.

\bibitem[Gr3]{Gr3}
M. Gromov,
\newblock {Asymptotic invariants of infinite groups},
Cambridge University Press, Geometric Group Theory, vol 2 (1993).



\bibitem[Sv]{Sv} \v Svarc, A.: The genus of a fibered space. {\em Trudy
Moskov. Mat. Ob\v s\v c} \textbf {10, 11} (1961 and 1962), 217--272,
99--126, (in {\em Amer. Math. Soc. Transl.} Series 2, vol
\textbf{55} (1966)).

\bibitem[Yu]{Yu} G. Yu {\em The Novikov conjecture for groups with finite
asymptotic dimension}, Ann. of Math vol 147, (1998) no. 2,
325-355.

\end{thebibliography}
\end{document}